\documentclass[12pt]{amsart}
\usepackage[english]{babel}
\title{The symmetric function theorem via the Fa\`a di Bruno formula}
\author[S. Van Hille]{Siegfried Van Hille}
\address{Fields Institute, 222 College Street, M5T 3J1 Toronto, Ontario, Canada}
\email{svanhill@fields.utoronto.ca}

\usepackage[margin = 3cm]{geometry}
\usepackage{amsfonts,amsmath,amsthm,amssymb,amsopn}
\usepackage{xcolor}
\usepackage{enumitem}
\usepackage{parskip}
\usepackage[colorlinks]{hyperref}
\hypersetup{
	linkcolor = {blue},
	citecolor = {blue},
	urlcolor = {blue},
}
\urladdr{\href{https://sites.google.com/view/siegfriedvanhille/academic}{https://sites.google.com/view/siegfriedvanhille/academic}}

\DeclareMathOperator{\N}{\mathbb N}

\DeclareMathOperator{\R}{\mathbb R}

\theoremstyle{plain}
\newtheorem{theorem}{Theorem}[section]
\newtheorem{corollary}[theorem]{Corollary}
\newtheorem{lemma}[theorem]{Lemma}
\newtheorem{proposition}[theorem]{Proposition}
\newtheorem*{theorem*}{Theorem}
\newtheorem*{corollary*}{Corollary}
\newtheorem*{proposition*}{Proposition}
\newtheorem*{conj*}{Conjecture}

\theoremstyle{definition}

\newtheorem*{conv*}{Convention}
\newtheorem{remark}[theorem]{Remark}
\newtheorem{ex}[theorem]{Example}

\makeatletter
\@namedef{subjclassname@2020}{%
  \textup{2020} Mathematics Subject Classification}
\makeatother

\begin{document}

\begin{abstract}
The symmetric function theorem states that a polynomial that is invariant under permutation of variables, is a polynomial in the elementary symmetric polynomials. We deduce this classical result, in the analytic setting, from the multivariate Fa\`a di Bruno formula. In two variables, this allows us to completely determine all coefficients that occur in the inductive equations.
\end{abstract}

\thanks{This publication was supported by the Fields Institute for Research in Mathematical Sciences. Its contents are solely the responsibility of the authors and do not necessarily represent the official views of the Institute. \href{http://www.fields.utoronto.ca}{Fields Institute for Research in Mathematical Sciences} \\ The author would like to thank L\'eo Jiminez for discussing the interesting combinatorial phenomenon we encountered, especially for writing down a formula for the map $\phi$.}



\maketitle

The fundamental theorem of symmetric function states that a polynomial in $n$ variables that is invariant under permuting its variables, is itself a polynomial in the so-called \emph{elementary symmetric polynomials}. To be precise, given a nonzero $n \in \N$, the $n$ elementary symmetric polynomials $\sigma^1, \ldots, \sigma^n$ are defined by:

\begin{align*}
&& \sigma^k(x) = \sum_{i_1 < \ldots < i_k} x_{i_1} \cdots x_{i_k} && (1 \leq k \leq n).
\end{align*}

In particular, $\sigma^1(x) = x_1 + \ldots + x_n$ and $\sigma^n(x) = x_1 \cdots x_n$. We set $\sigma = (\sigma^1, \ldots, \sigma^n)$. 

For a permutation $\pi \in \mathbb S^n$, let $\pi(x) = (x_{\pi^{-1}(1)}, \ldots, x_{\pi^{-1}(n)})$ be the usual action of $\mathbb S^n$ on a tuple. A function $f$ is \emph{symmetric} if we have that $f(\pi(x)) = f(x)$ for all $\pi \in \mathbb S^n$. Clearly, the elementary symmetric polynomials are symmetric. The symmetric function theorem states that the elementary symmetric polynomials are the building blocks for all analytic functions satisfying this property.

\begin{theorem*}[Symmetric function theorem]
Let $f(x): V \to \R$ be an analytic function, where $V \subset \R^n$ is an open neighborhood of the origin. If $f$ is symmetric, then there exists a unique analytic function $g$ such that $f = g \circ \sigma$. Moreover, if $f$ is a polynomial of degree $d$, then $g$ is a polynomial of degree at most $d$.
\end{theorem*}

The theorem is usually just stated for polynomials. It is classical and there are many proofs available, often with induction on the degree and number of variables. We refer to \cite{symmsurvey} for references, history and a different proof of the theorem.

We will give a direct proof, without using induction, but will use a technical combinatorial result to do so (the multivariate Fa\`a di Bruno formula, see Proposition \ref{faa}). We believe that it is interesting that this special application of the Fa\`a di Bruno formula naturally implies some combinatorial features the other proofs also have, such as a relation between the degree $d$ monomials of $f$ and monomials of weight $d$ of $g$. This short paper is written in such a way that the required combinatorial tools naturally follow from the proof strategy.

We will show that the coefficients of the power series representation of $f$ uniquely determine those of $g$ on $\sigma(V)$. This is sufficient in the analytic case. The theorem is still true in other case, for example smooth functions and Carleman-Denjoy classes, see \cite{symmsmooth}, \cite{symmCr} and \cite{symmgevrey}. These are questions of analysis and are in that sense independent of finding the function $g$.

The paper is structured as follows: in Section \ref{sec1} we introduce notation and the combinatorial proposition on which our proof is base. Then, we prove the symmetric function theorem in Section \ref{sec2}. Consequently, we give some remarks and examples in Section \ref{sec3}.  Finally, we make the system completely explicit in the two variable case in Section \ref{sec4}.

\section{Notation and the Fa\`a di Bruno formula} \label{sec1}

We first settle the notation that will be used in the proof of the main theorem and that is necessary to state the main combinatorial result that we will use to prove the symmetric function theorem.

Let $\nu \in \N^n$, then the \textbf{degree} of $\nu$ is defined by $$|\nu| = \sum_{i=1}^n \nu_i,$$ and the \textbf{weight} of $\nu$ is: $$||\nu|| = \sum_{i=1}^n i \nu_i.$$ For $\nu, \mu \in \N^n$, we define the equivalence relation $\nu \sim \mu$ as follows: $$ \nu \sim \mu \iff \exists \pi \in \mathbb S^n: \nu = \pi(\mu).$$

Furthermore, we define:

\begin{align*}
x^\nu &= \prod_{i=1}^n x_i^{\nu_i}, \\
\nu! &= \prod_{i=1}^n \nu_i!,
\end{align*}

with $0! = 1$ and $0^0 = 1$.

Finally, we use introduce the notation of \cite[Section 3]{faa}, although the roles of $f$, $g$ and $h$ are permuted to fit the notation here. Proposition \ref{faa} below gives an explicit formula to compute the coefficients $f_\nu$ of a series $$f(x) = \sum_{\nu \in \N} f_\nu x^\nu,$$ which is the composition of a series $$g(x) = \sum_{\nu \in \N^n} g_\nu x^\nu$$ and $$h(x) = (h^1(x), \ldots, h^n(x)),$$ where for $1 \leq i \leq n$: $$h_i(x) = \sum_{\nu \in \N^n} \frac{h^i_\nu}{\nu!}x^\nu.$$

We denote $h_\nu = (h^1_\nu, \ldots, h^n_\nu)$. Since in our setting $h$ will be the polynomial map $\sigma$, and $\sigma(0) = 0$, we assume that $h$ is polynomial and $h(0) = 0$. Consequently, we can use \cite[Formula 3.7]{faa}, which is the following formula.

\begin{proposition}[{\cite[Theorem 3.1]{faa}}] \label{faa}
Let $f = g \circ h$ be as above, $\nu \in \N^n$ be nonzero and set $n = |\nu|$. Then we have:
\begin{equation} \label{faaeq}
f_\nu = \sum_{1 \leq |\lambda| \leq |\nu|} g_\lambda \lambda! \sum_{p(\nu,\lambda)} \prod_{j=1}^{n} \frac{ (h_{l_j})^{k_j} }{ k_j! (l_j!)^{|k_j|} },
\end{equation}
where
\begin{multline*}
p(\nu,\lambda) = \{ (k_1, \ldots, k_{n}; l_1, \ldots, l_n) \mid \text{ for some } 1 \leq s \leq n, \\
 k_i = 0 \text{ and } l_i = 0 \text{ for } 1 \leq i \leq n-s; |k_i| > 0 \text{ for } n-s+1 \leq i \leq n; \\
\text{ and } 0 \prec l_{n-s+1} \prec \ldots \prec l_n \text{ are such that } \\
\sum_{i=1}^n k_i = \lambda \text{ and } \sum_{i=1}^n |k_i| l_i = \nu \},
\end{multline*}
with $\nu \prec \mu$ if and only if $|\nu| < |\mu|$ or, if $|\nu| = |\mu|$, then $\nu$ comes before $\mu$ lexicographically.
\end{proposition}

\begin{remark}
In fact, this is a direct consequence of the multivariate chain rule for arbitrary order derivatives, also known as the Fa\`a di Bruno formula. In the set $p(\nu,\lambda)$, one should consider the $k_i$ as multiplicities and $l_i$ as the order of differentiation to better understand the two relations they are subject to. The sparse support of $\sigma$, together with these relations, will be crucial in our proof of the main theorem.
\end{remark}

\section{Proof of the theorem} \label{sec2}

This entire section is the proof of the main theorem. So, suppose $f: V \subset \R$ is analytic, where $V$ is an open neighborhood of the origin. The idea is straightforward: for each $\nu \in \N^n$, we use Formula \ref{faaeq} to get an equation:

\begin{equation} \label{system}
f_\nu = \sum_{1 \leq |\lambda| \leq |\nu|} g_\lambda \lambda! \sum_{p(\nu,\lambda)} \prod_{j=1}^{n} \frac{ (\sigma_{l_j})^{k_j} }{ k_j! (l_j!)^{|k_j|} }. \tag{$\star$}
\end{equation}

This gives us infinitely many \emph{linear} equations in the variables $g_\lambda$. Our goal is to show that there is a unique solution. Note that for $\nu = (0, \ldots, 0)$, Formula \ref{faa} doesn't apply, but in that case we have $g_{(0, \ldots 0)} = f_{(0, \ldots 0)}$, of course.

Firstly one checks that, since the function $f$ and map $\sigma$ are symmetric, both sides are invariant under the action of $\mathbb S^n$. Hence, we only have to consider one equation for each equivalence class. 

Now, due to the specific form of the map $\sigma$, we expect many $g_\lambda$ to have a zero coefficient. This is what we show in the next lemma.

\begin{lemma}
Let $\nu \in \N^n$ be nonzero and let $d = |\nu|$. In equation (\ref{system}), one can replace $\{1 \leq |\lambda| \leq |\nu| \}$ by $\{ || \lambda || = d \}$.
\end{lemma}
\begin{proof}
We will show that if an element of $p(\nu,\lambda)$ yields a nonzero contribution $$ \prod_{j=1}^{n} \frac{ (\sigma_{l_j})^{k_j} }{ k_j! (l_j!)^{|k_j|} }$$ to the coefficient of $g_\lambda$, then $||\lambda|| = d$. Consequently, if $||\lambda|| \neq d$, then its coefficient is zero, so we may indeed replace the summation of $\lambda$ as stated.

To prove this, one plays with the simple support of the component functions of $\sigma$ and the relations the $k_i$ and $l_i$ satisfy (see Theorem \ref{faa}). To be precise, let us fix an element of $p(\nu,\lambda)$, that is $k_1, \ldots, k_n$, $l_1, \ldots, l_n$ and some $1 \leq s \leq n$. Now, since the first $n-s$ of these are zero, we discard them and reindex. So we may assume we have $k_1, \ldots, k_s$, nonzero, and $0 \prec l_1 \prec \ldots \prec l_s$. 

Now, for $1 \leq k \leq n$, note that the support of $\sigma^k$ is exactly all $\mu \in \{0,1\}^n$ with $|\mu| = k$. In particular, the supports of the component functions of $\sigma$ are pairwise disjoint. Consequently, to obtain a nonzero contribution, each $k_j$ must be of the form $(0, \ldots, 0, t_j, 0, \ldots, 0)$, for some $t_j \in \N$. We say that $k_j$ ``selects" a certain component function of $\sigma$. Note that if $k_j$ selects $\sigma^k$, then $|l_j| = k$, or the corresponding term is zero. Now, since $$\sum_{j=1}^s k_i = \lambda, $$ and because of the special form of the $k_j$, we have that
\[
\sum_{k_j \text{ selects } \sigma^k} |k_j| = \lambda_k.
\]
Then we put everything together to find that:
\[
|\nu| = \sum_{j=1}^s |k_j| |l_j| = \sum_{i = 1}^n \sum_{k_j \text{ selects } \sigma^i} |k_j| |l_j| = \sum_{i = 1}^n i \sum_{k_j \text{ selects } \sigma^i} |k_j| = \sum_{i=1}^n i \lambda_i =  ||\lambda||.
\]
\end{proof}

\begin{remark}
We will frequently use the reindexing of an element of $p(\nu,\lambda)$, as in the proof above, without mentioning it.
\end{remark}

Thus, given $d \in \N$, only the $g_\lambda$ with $\lambda$ of weight $d$ appear in the equations (\ref{system}) corresponding to $\nu$ of degree $d$. So the natural question we ask ourselves is: do we have sufficiently many equations to determine them? Before proving that this is the case, let us consider some elementary examples of equations, which might be handy to understand the set $p(\nu,\lambda)$ a bit better.

\begin{ex} \hspace{1cm} \label{observation}
\begin{enumerate}
\item Let us first consider the (equivalence class) of $(d, 0, \ldots, 0)$, which corresponds to $x_i^d$, for some $1 \leq i \leq n$. Arguably, these are the simplest monomials. Now it is easy to see that there is only one element of $p(\nu,\lambda)$ that yields a nonzero contribution, namely: $$k_1 = (d, 0, \ldots, 0) \text{ and } l_1 = (1, 0, \ldots, 0).$$ This corresponds to the fact that there is only one way to obtain $x_i^d$ using the elementary symmetric polynomials, which is using $\sigma^1(x)^d$. Thus, for each $d \in \N$ we find the equation: $$f_{(d, 0, \ldots, 0)} = g_{(d, 0, \ldots, 0)}.$$

\item Now let us look at the second least complicated monomial, say $x_1^{d-1}x_2$, which corresponds to $\nu = (d-1,1, 0, \ldots, 0)$. Now one has two elements in $p(\nu,\lambda)$ which yield a nonzero contribution, namely:
\begin{align*}
k_1 = (d-1, 0, \ldots, 0) &\text{ and } l_1 = (d-1, 0, \ldots, 0), \\
k_2 = (1, 0, \ldots, 0) &\text{ and } l_2 = (0, 1, 0, \ldots, 0),
\end{align*}
in which case $\lambda = (d, 0, \ldots, 0)$, and
\begin{align*}
k_1 = (d-2, 0, \ldots, 0) &\text{ and } l_1 = (d-2, 0, \ldots, 0), \\
k_2 = (0, 1, \ldots, 0) &\text{ and } l_2 = (1, 1, 0, \ldots, 0),
\end{align*}
where $\lambda = (d-2,1, 0, \ldots, 0)$.

\item Note that one can write any $\nu \in \N^n$ using $\lambda = (|\nu|, 0, \ldots, 0)$, by setting:
\begin{align*}
k_1 = (\nu_1, 0, \ldots, 0) &\text{ and } l_1 = (1, 0, \ldots, 0), \\
k_2 = (\nu_2, 0, \ldots, 0) &\text{ and } l_2 = (0, 1, 0, \ldots, 0), \\
&\;\;\;\vdots \\
k_n = (\nu_n, 0, \ldots, 0) & \text{ and } l_n = (0, \ldots, 0, 1).
\end{align*}
In other words, we will find the variable $g_\lambda$ with $\lambda = (|\nu|, 0, \ldots, 0)$ in each equation.
\end{enumerate}
\end{ex}

It is advised to the reader to make a few more examples. What one might hope, after considering $(1)$ and $(2)$ above, is that one obtains a triangular system, and this is indeed the case.  But first, we will show that it has equally many equations as variables.

\begin{lemma}\label{square}
Let $d \in \N$. Then we have a bijection between the sets $$\{ \nu \in \N^n \mid |\nu| = d \} / \sim$$ and $$\{ \nu \in \N^n \mid ||\nu|| = d\}.$$
\end{lemma}

\begin{proof}
Each equivalence class has a \emph{unique} representative $\nu \in \N^n$ such that $$\nu_1 \geq \nu_2 \geq \ldots \geq \nu_n.$$ Then define $\phi(\nu) = (\nu_1 - \nu_{2}, \nu_{2} - \nu_{3}, \ldots, \nu_{n-1} - \nu_n, \nu_n) \in \N^n$. An easy computation shows that $||\phi(\nu)|| = d$, and it is straightforward to check that $\phi$ is bijective.
\end{proof}

As a consequence, we now have a square linear system for each $d \in \N$, consisting of one equation for each element of the set $\{ \nu \in \N^n \mid |\nu| = d\} / \sim$ in the variables $\{g_\lambda \mid \lambda \in \N^n, ||\lambda|| = d\}$. So all we have to do, is to solve them. As observed below Example \ref{observation}, we will show that each system is triangular.

\begin{lemma}
For each nonzero $d \in \N$, the square linear system defined by the equations corresponding to each element in $\{ \nu \in \N^n \mid |\nu| = d\} / \sim$ is triangular and its diagonal is identically $1$.
\end{lemma}
\begin{proof}
We order the set $\{ \nu \in \N^n \mid |\nu| = d\} / \sim$ by the reverse lexicographical order on the unique representative satisfying $\nu_1 \geq \nu_2 \geq \ldots \geq \nu_n$, so the equivalence class of $(d,0, \ldots 0)$ is the smallest element. We will show that going to the next equivalence class, one finds exactly one new $g_\lambda$ and $\lambda$ is given by the map $\phi$ of Lemma \ref{square}. In this way, the system is lower triangular. For simplicity, we will identify each equivalence class with its unique representative as before, and the unknown $g_\lambda$ with its index $\lambda$.

The key is to understand the map $\phi$ better. Let us explain how $\phi$ constructs a particular element of $p(\nu,\lambda)$. The function $\phi$ does the following: given $\nu$, with $\nu_1 \geq \nu_2, \ldots \geq \nu_n$, one first uses $\nu_n$ times the vector $(1, \ldots, 1)$ to decompose $\nu$, so we obtain: $$\nu = (\nu_1 - \nu_n, \ldots, \nu_{n-1} - \nu_n, 0) + \nu_n(1, \ldots, 1).$$ Consequently, we can use $\nu_{n-1}-\nu_n$ times $(1, \ldots, 1, 0)$ and continue this proces. In this way, we have written $\nu = \sum_{i = 1}^n |k_i| l_i$, where $|k_i| = \nu_i - \nu_{i+1}$ (define $\nu_{n+1} = 0$), $l_i = (1, \ldots, 1, 0, \ldots, 0)$ ($i$ entries are $1$), and $\sum_{i=1}^n k_i = \lambda = \phi(\nu)$.

Now, for the smallest element, namely $(d, 0, \ldots, 0)$, we have seen in Example \ref{observation} that there is only one element in $p(\nu,\lambda)$, where $\lambda = (d, 0, \ldots, 0) = \phi(\nu)$. Now, it is easy to see that if $\mu < \nu$, then one cannot decompose $\mu$ according to $\phi(\nu)$. Now the proof follows from a pigeon hole type argument: if one goes to the next $\nu$, one finds a least one new $\lambda$, namely $\phi(\nu)$. Since we have a square system, it cannot be more than one.

To see that the diagonal is identically one, note that there is only one element in $p(\nu,\lambda)$ with $\lambda = \phi(\nu)$ (it is constructed above), that it yields the coefficient $1$.
\end{proof}

As a final remark: if $f$ was a polynomial of degree $d$, the coefficients of $f$ determine all coefficients of a polynomial $g$ with weights at most $d$. In particular, the degree of $g$ is as most $d$. In the real analytic case, since $f$ is the uniform limit of its Taylor polynomials on $K$, $g$ is the uniform limit of the polynomials obtained from the Taylor polynomials via the symmetric function theorem.

\section{Examples and remarks} \label{sec3}

We start by giving some remarks, mostly concerning the coefficients of the linear systems involved in the symmetric function theorem. In particular, we comment on how they have a combinatorial interpretation.

\begin{remark}
In \cite[p. 22-23]{symmsurvey} one discusses the use of the lexicographical order in proofs of the symmetric function theorem. We think that it is elegant that we somehow deduce the ``symmetric" lexicographical order from our proof method here. Example \ref{observation} indicates that we should order the equations with respect to the lexicographical order, although on the unique representative of each class, as in the proof of Lemma \ref{square}. We think this is quite natural, especially after writing down some examples of equations. In this way, one also observes that the ``new $\lambda$" that occurs in the next equation, is given by the map $\phi$.
\end{remark}

\begin{remark}
In many proofs, one initially assumes $f$ to be homogenous of degree $d$, by considering each degree separately. This corresponds to the distinct linear systems for each $d \in \N$ here.
\end{remark}

\begin{remark}
Note that in (\ref{system}) we don't have that $\lambda! \prod_{j=1}^s (1/k_j) = 1$ in general, despite the particular form of the $k_j$'s. This is due to the fact that it might happen that one picks different $l_j$'s from the same $\sigma^k$. For example, see Example \ref{observation} $(3)$.
\end{remark}

\begin{remark}
In theory, the coefficients of equation (\ref{system}) are explicit here, in contrast to the proofs based on induction (of course, assuming the non-trivial Proposition \ref{faa}). Usually, the interpretation of $$\nu! \prod_{j=1}^n \frac{1}{k_j! (l_j)^{|k_j|}}$$ is related to partitions of $\nu$. In the univariate case, it is the number of partitions of a set with $\nu$ elements into $k_1$ sets with $l_1$ elements \ldots, and are related to Stirling numbers of the second kind (see \cite[Formula 1.2]{faa}). More precisely, given a $\lambda \in \N$, summing all of these gives the number of partitions of $\nu$ into $\lambda$ nonempty sets. Of course, we work with tuples here, but the interesting part is that we only count specific kind of partitions, which the authors think of as binary decompositions of $\nu$.
\end{remark}

\begin{remark}
In our case, since the nonzero contributions in (\ref{system}) have $l_j \in \{0,1\}^n$, it follows that $(l_j!)^{|k_j|} = 1$. Of course, if one would write the power series $h(x) = \sum h_\nu/\nu! x^\nu$ in Theorem \ref{faa} as $\sum h'_\nu x^\nu$, one can also ignore the $l_j$'s. The point of writing it in this way is to give the combinatorial interpretation above. Nevertheless, in our case only the $k_j$'s play a role.
\end{remark}

\begin{remark} \label{column1}
In Example \ref{observation} $(3)$, we actually made the first column of each system explicit. Let $d \in \N$ and $\nu \in \N^n$ of degree $d$. To compute the coefficients of the first column, one has to compute the nonzero contributions of elements of $p(\nu,\lambda)$ with $\lambda = (d, 0, \ldots, 0)$. Now, there is only one such element, given in Example \ref{observation} $(3)$, thus the coefficient of $g_{(d,0, \ldots, 0)}$ in the equation corresponding to $\nu$ is:
\[
\frac{\lambda!}{\prod_{j=1}^n k_j! (l_j!)^{|k_j|}}  = \frac{d!}{\nu!}.
\]
\end{remark}

By the preceding remarks, the system has the following form, for $n \geq 3$.
\begin{equation}
\begin{pmatrix}
1 &  &  &  &  &  &  &  &  &  &  \\
 & 1 &  &  &  &  &  &  &  &  &  \\
 & 2 & 1 &  &  &  &  &  &  & & \\
 &  &  & 1 &  &  &  &  &  & & \\
 &  &  & 3 & 1 &  &  &  &  & & \\
  &  &  & 6 & \ast & 1 &  &  &  &  & \\
 &  &  &  &  &  & 1 &  &  &  &  \\
 &  &  &  &  & & 4 & 1 &  &  & \\
 &  &  &  &  & & 6 &  \ast & 1 &  & \\
 &  &  &  &  & & 12 &  \ast & \ast & 1 & \\
 &  &  &  &  & &  &  &  &  & \ddots
\end{pmatrix}
\begin{pmatrix}
g_{(1, 0, \ldots, 0)} \\
g_{(2, 0, \ldots, 0)} \\
g_{(0,1, 0, \ldots, 0)} \\
g_{(3,0, 0, \ldots, 0)} \\
g_{(1,1, 0, \ldots, 0)} \\
g_{(0,0,1,0, \ldots, 0)} \\
g_{(4, 0, \ldots, 0)} \\
g_{(2,1, 0, \ldots, 0)} \\
g_{(0,2,0, \ldots, 0)} \\
g_{(1, 0, 1, 0, \ldots, 0)} \\
\vdots
\end{pmatrix}
=
\begin{pmatrix}
f_{(1, 0, \ldots, 0)} \\
f_{(2, 0, \ldots, 0)} \\
f_{(1,1, 0, \ldots, 0)} \\
f_{(3,0, \ldots, 0)} \\
f_{(2,1, 0, \ldots, 0)} \\
f_{(1,1,1, 0, \ldots, 0)} \\
f_{(4,0, \ldots, 0)} \\
f_{(3,1, 0, \ldots, 0)} \\
f_{(2,2, 0, \ldots, 0)} \\
f_{(2,1,1, 0, \ldots, 0)} \\
\vdots
\end{pmatrix}
\end{equation}

For general $d,n \in \N$, it seems more difficult to compute the $\ast$ explicitly. For example, in the above, consider $\nu = (2,1,1, 0, \ldots, 0)$ and $\lambda = (2, 1, 0, \ldots, 0)$. The problem is that there are several elements in $p(\nu,\lambda)$ now, depending on which $l_j$ is selected from $\sigma^2$, i.e. $(1,1,0, 0 \ldots, 0)$, $(1,0,1,0,\ldots,0)$ or $(0,1,1,0,\ldots,0)$. Moreover, the coefficient corresponding to the choice is not always the same. In the first case, which is the element
\begin{multline*}
k_1 = (1, 0, 0 \ldots, 0), l_1 = (1, 0, 0, 0 \ldots, 0), k_2 = (1, 0, 0, \ldots, 0), l_2 = (0,0,1,0, \ldots, 0) \\ k_3 = (0,1,0, \ldots, 0), l_3 = (1,1,0, 0, \ldots, 0)
\end{multline*}
of $p(\nu,\lambda)$, the coefficient is $\lambda!/(k_1!k_2!k_3!) = 2!/1!1!1! = 2.$ However, in the last case, which is the element
\begin{equation*}
k_1 = (2, 0, 0 \ldots, 0), l_1 = (1, 0, 0, 0 \ldots, 0), k_2 = (0, 1, 0, \ldots, 0), l_2 = (0,1,1,0, \ldots, 0)
\end{equation*}
of $p(\nu,\lambda)$, its coefficient is $1$. However, this ``choice phenomenon" cannot happen if $n = 2$. We will make the system completely explicit in the next section.

Finally, we give some examples.

\begin{ex} \hspace{1cm}
\begin{enumerate}
\item Let $f(x_1,x_2) = x_1 + x_2 + 3x_1^2 + 3x_2^2 -5x_1x_2$. There is only one equation for $d = 1$, which is $f_{10} = 1 = g_{10}$. For $d = 2$, we find the system:
\begin{align*}
f_{20} = 3 &= g_{20}, \\
f_{11} =-5 &= 2g_{20} + g_{01}.
\end{align*}
Thus $g_{20} = 3$ and $g_{01} = -11$. Since there are no $\nu$ of higher degree, we have found $g$, and, indeed, we see that
\[
g(\sigma(x)) = (x_1+x_2) +3(x_1+x_2)^2 -11x_1x_2 = f(x_1,x_2),
\]
where $g(x_1,x_2) = x_1 + 3x_1^2 - 11x_2$.

\item Let $f(x_1,x_2,x_3) = 3x_1x_2x_3 - x_1x_3^2 - x_1^2x_3 - x_2x_3^2 - x_2^2x_3 - x_1x_2^2 - x_1^2x_2 + x_1 + x_2 + x_3$. Again, for $d = 1$, we simply have $g_{100} = 1 = f_{100}$. The equations for $d = 2$ give the trivial solution, since $f$ has no monomials of degree $2$. We can find the other coefficients using the degree $3$ equations:
\begin{align*}
f_{300} &= g_{300} \\
f_{210} &= 3g_{300} + g_{110} \\
f_{111} &=  6g_{300} +  3 g_{110} +  g_{001} \\
\end{align*}
Yielding consecutively:
\begin{align*}
g_{300} &= 0 \\
g_{110} &= -1\\
g_{001} &= 3 - 3(-1) = 6 \\
\end{align*}
So we obtain $g(x_1,x_2,x_3) = x_1 - x_1x_2 + 6x_3$.

\item Consider the function $f(x_1,x_2) = e^{x_1+x_2}$. Now clearly, $g(x_1,x_2) = e^{x_1}$. But now we can compute the Taylor expansion of $f$, using the simpler expansion of $g$ (so we reverse the problem), which is the well known:
\[
g(x_1,x_2) = \sum_{i \in \N} \frac{1}{i!}x_1^i.
\]
For degree $1$, we have $f_{10} = g_{10} = 1$. For $d = 2$, we solve:
\begin{align*}
f_{20} &= g_{20} = 1, \\
f_{11} &= 2g_{20} + g_{01} = 2,
\end{align*}
and for $d = 3$:
\begin{align*}
f_{30} &= g_{30} = 1 \\
f_{21} &= 3g_{30} + g_{11} = 3 \\
\end{align*}
Consequently, the first few terms of the Taylor expansion of $f(x_1,x_2)$ are: $$x_1 + x_2 + x_1^2 + 2x_1x_2 + x_2^2 + x_1^3 + 3x_1^2x_2 + 3x_1x_2^2 + x_2^3 +\ldots$$
Of course, one could verify this easily in this example.
\end{enumerate}
\end{ex}

\section{The system for $n = 2$} \label{sec4}

In this section we will make all coefficients in the systems explicit for $n =2$. The key will be Lemma \ref{unique}. The proof of the first lemma is obvious.

\begin{lemma} \label{alleq}
Let $d \in \N$. Then $(d,0) < (d-1,1) < \ldots < (\lceil d/2 \rceil, \lfloor d/2 \rfloor)$ are the unique representatives of $\{ \nu \in \N^2 \mid |\nu| = d\} / \sim$ that satisfy $\nu_1 \geq \nu_2$. 
\end{lemma}

Next we compute the elements of $p(\nu,\lambda)$.

\begin{lemma} \label{unique}
Let $\nu \in \N^2$ with $\nu_1 \geq \nu_2$, $\lambda \in \N^2$ with $||\lambda|| = |\nu|$. Then $p(\nu,\lambda)$ is empty if $\phi^{-1}(\lambda) > \nu$. If $p(\nu,\lambda)$ is not empty, then it has a unique element given by:
\[
k_1 = (\nu_1-\lambda_2, 0), l_1 = (1,0), k_2 = (\nu_2-\lambda_2, 0), l_2 = (0,1), k_3 = (0, \lambda_2), l_3 = (1,1).
\]
\end{lemma}

\begin{proof} 
Firstly, note that $\phi^{-1}(\lambda) \leq \nu$ simply means that the row index is higher or equal to the column index. We already know that each system is lower triangular, which followed from the fact that $p(\nu,\lambda)$ is indeed empty if $\phi(\lambda)>\nu$.

So, assume $\phi^{-1}(\lambda) \leq \nu$. We will show that a decomposition of $\nu$ of the form
\[
\nu = |k_1|(1,0) + |k_2|(0,1) + |k_3|(1,1),
\]
satisfying the conditions of Proposition \ref{faa}, necessarily forces the rest of the statement. We find immediately that $k_3 = (0,\lambda_2)$, since we need to have that $$\sum_{j=1}^3 k_j = \lambda.$$ Hence, rewriting the equation above, we find:
\[
\nu - \lambda_2(1,1) = (\nu_1 - \lambda_2, \nu_2 - \lambda_2)  = |k_1|(1,0) + |k_2|(0,1).
\]
Now, if $\lambda_2 > \nu_2$, this is impossible, but since $\phi^{-1}(\lambda) = (\lambda_1+\lambda_2, \lambda_2) \leq \nu$, it follows that $\nu_2 \geq \lambda_2$. Since we have $\nu_1 \geq \nu_2$, also $\nu_1 - \lambda_2 \geq 0$. The above equation then forces:
\begin{align*}
k_1 &= (\nu_1 - \lambda_2, 0), \\
k_2 &= (\nu_2 - \lambda_2, 0).
\end{align*}
\end{proof}

\begin{remark}
It might seem that in the lemma above, one should add that $\lambda_1 = |\nu| - 2\lambda_2$, but this follows from $||\lambda|| = |\nu|$.
\end{remark}

\begin{remark}
Note that the condition $\phi^{-1}(\lambda) \leq \nu$ is equivalent to $\lambda_2 - \nu_2 \geq 0$.
\end{remark}

\begin{corollary} \label{coeffn2}
The coefficient of $g_\lambda$ in the equation (\ref{system}) corresponding to $\nu \in \N^2$ is given by
\[
\frac{\lambda_1!}{(\nu_1-\lambda_2)!(\nu_2-\lambda_2)!},
\]
if $||\lambda|| = |\nu|$ and $\lambda_2 - \nu_2 \geq 0$, and $0$ otherwise.
\end{corollary}

Note that this agrees with our result in Remark \ref{column1}. Combining Lemma \ref{alleq} and Corollary \ref{coeffn2}, this system looks as follows for $d \in \N$:
\begin{equation*}
\begin{pmatrix}
1 &  &  &  \\
d & 1 &  &  \\
\vdots & \vdots & \ddots &  \\
\frac{d!}{(\lceil d/2 \rceil)!(\lfloor d/2 \rfloor)!} & \frac{(d-2)!}{(\lceil d/2 \rceil - 1)!(\lfloor d/2 \rfloor -1)!} & \cdots & 1
\end{pmatrix}
\begin{pmatrix}
g_{(d,0)} \\
g_{(d-2,1)} \\
\vdots \\
g_{(d \mod 2, \lfloor d/2 \rfloor)}
\end{pmatrix}
 = 
\begin{pmatrix}
f_{(d,0)} \\
f_{(d-1,1)} \\
\vdots \\
f_{(\lceil d/2 \rceil, \lfloor d/2 \rfloor)}
\end{pmatrix}.
\end{equation*}

\bibliographystyle{alpha}
\bibliography{symmthm}

\begin{thebibliography}{BSC17}

\bibitem[Bar72]{symmCr}
G\'{e}rard Barban\c{c}on.
\newblock Th\'{e}or\`eme de {N}ewton pour les fonctions de class {$C^{r}$}.
\newblock {\em Ann. Sci. \'{E}cole Norm. Sup. (4)}, 5:435--457, 1972.

\bibitem[Bro86]{symmgevrey}
M.~D. Bronshte\u{\i}n.
\newblock Representation of symmetric functions in {G}evrey-{C}arleman spaces.
\newblock {\em Zap. Nauchn. Sem. Leningrad. Otdel. Mat. Inst. Steklov. (LOMI)},
  149(Issled. Line\u{\i}n. Teor. Funktsi\u{\i}. XV):116--126, 189, 1986.

\bibitem[BSC17]{symmsurvey}
Ben Blum-Smith and Samuel Coskey.
\newblock The fundamental theorem on symmetric polynomials: history's first
  whiff of {G}alois theory.
\newblock {\em College Math. J.}, 48(1):18--29, 2017.

\bibitem[CS96]{faa}
G.~M. Constantine and T.~H. Savits.
\newblock A multivariate {F}a\`a di {B}runo formula with applications.
\newblock {\em Trans. Amer. Math. Soc.}, 348(2):503--520, 1996.

\bibitem[Gla63]{symmsmooth}
Georges Glaeser.
\newblock Fonctions compos\'{e}es diff\'{e}rentiables.
\newblock {\em Ann. of Math. (2)}, 77:193--209, 1963.

\end{thebibliography}

\end{document}